\documentclass{amsart}
\usepackage{amsmath,amssymb}

\input{xy}
\xyoption{all}

\newcommand{\w}{\omega}
\newcommand{\Comp}{\mathbf{Comp}}
\newcommand{\C}{\mathbf{C}}
\newcommand{\T}{\mathbf{T}}

\newcommand{\id}{\mathrm{id}}
\newcommand{\supp}{\mathrm{supp}}
\newcommand{\II}{\mathbb I}

\newtheorem{theorem}{Theorem}
\newtheorem{lemma}{Lemma}
\newtheorem{corollary}{Corollary}

\title{On monomorphic topological functors with finite supports}

\author[T.Banakh, M.Martynenko, M.Zarichnyi]{Taras Banakh, Marta Martynenko, and
        Mykhailo Zarichnyi}
\address{Department of Mathematics, Ivan Franko National University of Lviv, Ukraine}
\email{t.o.banakh@gmail.com, martamartynenko@ukr.net, mzar@litech.lviv.ua}

\begin{document}

\begin{abstract}
We prove that a monomorphic functor $F:\Comp\to\Comp$ with finite supports is epimorphic, continuous, and its maximal $\emptyset$-modification $F^\circ$ preserves intersections. This implies that a monomorphic functor $F:\Comp\to\Comp$ of finite degree $\deg F\le n$ preserves (finite-dimensional) compact ANRs if the spaces $F\emptyset$, $F^\circ\emptyset$, and $Fn$ are finite-dimensional ANRs. This improves a known result of Basmanov.
\end{abstract}
\maketitle

\section{Introduction}

In this paper we study monomorphic functors with finite supports defined on topological categories and then apply the obtained results to generalize the classical result of Basmanov on the preservation
of (finite-dimensional) compact ANRs by functors of finite degree in the category $\Comp$ of compact Hausdorff spaces and their continuous maps.

Let $\T$ denote the category whose objects are topological spaces and whose morphisms are (not necessarily continuous) functions between topological spaces.
By a {\em topological category} we understand a subcategory $\C$ of the category $\T$ such that each finite discrete topological space is an object of $\C$ and each map $f:D\to X$ from a finite discrete space to an object of the category $\C$ is a morphism of $\C$. This implies that each monomorphism of the category $\C$ is an injective function.

We say that a functor $F:\C\to\T$ defined on a topological category $\C$
\begin{itemize}
\item is {\em monomorphic} if $F$ preserves monomorphisms;
\item has {\em finite supports} (resp. {\em finite degree} $\le n$) if for each object $X$ of $\C$ and each $a\in FX$ there is a map $f:A\to X$ from a finite discrete space $A$ (of cardinality $|A|\le  n$) such that $a\in Ff(FA)$;
\item {\em preserves the empty set} if $F\emptyset=\emptyset$.
\end{itemize}

Let us observe that for each (monomorphic) functor $F:\C\to\T$ that does not preserve the empty set we can change the value of $F$ at $\emptyset$ and define a new (monomorphic) functor
$F_\circ:\C\to\T$,
$$F_\circ X=\begin{cases}FX&\mbox{if $X\ne\emptyset$},\\
\emptyset&\mbox{if $X=\emptyset$},
\end{cases}$$
which preserves the empty set. This functor $F_\circ$ is called the {\em minimal $\emptyset$-modification} of $F$.

 By an {\em $\emptyset$-modification} of a (monomorphic) functor $F:\C\to\T$ we understand a (monomorphic) functor $\tilde F:\C\to\T$ such that $\tilde FX=FX$ for each non-empty object $X$ of the category $\C$. So, the values of the functors $F$ and $\tilde F$ can differ only on the empty set.
The functor $F_\circ$ is the minimal $\emptyset$-modification of $F$ in the sense that $F_\circ$ is a subfunctor of any $\emptyset$-modification $\tilde F$ of $F$.

It turns out that the family of all $\emptyset$-modifications of a given monomorphic functor $F$ has a maximal element $F^\circ$. Below we identify a finite ordinal $n$ with the finite discrete space $\{0,\dots,n-1\}$. For $i\in 2$ let $f_i:1\to\{i\}\subset 2$ be the constant map.

\begin{theorem}\label{t1} Each monomorphic functor $F:\C\to\T$ has the maximal $\emptyset$-modification $F^\circ:\C\to\T$ assigning to $\emptyset$ the space $$F^\circ\emptyset=\{a\in F1:Ff_0(a)=Ff_1(a)\}\subset F1.$$
\end{theorem}

\begin{proof} In the formulation we have defined the action of the functor $F^\circ$ on the empty set. For each non-empty space $X$ in $\C$ we put $F^\circ X=FX$.

Now we define the action of $F^\circ$ on morphisms.
Let $f:X\to Y$ be a morphism of the category $\C$.
If $X$ is not empty, then so is $Y$ and we put $F^\circ f=Ff$. If $X=\emptyset=Y$,  then $F^\circ f$ is the identity map of the space $F^\circ\emptyset$. If $X=\emptyset$ and $Y\ne\emptyset$, then we put $$F^\circ f=Fg|F^\circ\emptyset:F^\circ\emptyset\to F^\circ Y=FY$$ where $g:1\to Y$ is any map.

Let us check that the morphism $F^\circ f$ is well-defined, i.e., it does not depend on the choice of the map $g:1\to Y$. Indeed, given another map $g':1\to Y$, consider the map $h:2\to Y$ defined by $h(0)=g(0)$ and $h(1)=g'(0)$. It follows that $g=h\circ f_0$ and $g'=h\circ f_1$ and then for any $a\in F^\circ\emptyset$
$$Fg(a)=F(h\circ f_0)(a)=Fh\circ Ff_0(a)=Fh\circ Ff_1(a)=F(h\circ f_1)(a)=Fg'(a).$$

This argument also implies that $F^\circ(f\circ g)=F^\circ f\circ F^\circ g$ for any morphisms $$\xymatrix@1{X\ar[r]^f&Y\ar[r]^g&Z}$$ of the category $\C$. This means that $F^\circ:\C\to\T$ is a well-defined monomorphic functor. It is clear that $F^\circ$ is an $\emptyset$-modification of $F$.
\smallskip

It remains to check that $F^\circ$ is the maximal $\emptyset$-modification of $F$. We shall show that for any $\emptyset$-modification $\tilde F$ of $F$ we get $\tilde F i^\emptyset_1(\tilde F\emptyset)\subset F^\circ\emptyset\subset F1$ where $i^\emptyset_1:\emptyset\to 1$ is the unique map.  Applying the functor $\tilde F$ to the equality $f_0\circ i^\emptyset_1=f_1\circ i^\emptyset_1$  we get $\tilde Ff_0\circ \tilde F i^\emptyset_1(a)=\tilde Ff_1\circ \tilde Fi^\emptyset_1(a)$ for every $a\in\tilde F\emptyset$, which means that  $\tilde Fi^\emptyset_1(a)\in F^\circ\emptyset$ and thus $\tilde Fi^\emptyset_1(\tilde F\emptyset)\subset F^\circ\emptyset$.
\end{proof}

Now, given a functor $F:\C\to\T$ with finite supports and an object $X$ of the category $\C$, we define the support map $\supp_X:F^\circ X\to [X]^{<\w}$ into the set $[X]^{<\w}$ of finite subsets of $X$. Each finite subset $A\subset X$ will be endowed with the discrete topology. By $i^A_X:A\to X$ we denote the identity map from the finite discrete space $A$ to $X$.

 For an element $a\in F^\circ X$ the set
$$\supp_X(a)=\bigcap\{A\in[X]^{<\w}:a\in F^\circ i^A_X(F^\circ\! A)\}$$
is called the {\em support} of $a$.

The principal result of this paper is the following:

\begin{theorem}\label{support} Let $\C$  be a topological category and $F:\C\to\T$ be a monomorphic functor with finite supports. For any element $a\in F^\circ X$ the support $A=\supp_X(a)$ is a well-defined finite subset of $X$ such that $a\in F^\circ i^A_X(F^\circ A)$.
\end{theorem}

We postpone the proof of this theorem till Section~\ref{s2}. Now we discuss an  application of Theorem~\ref{support} to functors of finite degree in the category $\Comp$ of compact Hausdorff spaces and their continuous maps.
First we recall the necessary definitions, see \cite{TZ} for more details.

A functor $F:\Comp\to\T$
\begin{itemize}
\item is {\em epimorphic} if $F$ preserves epimorphisms (which coincide with surjective maps in the categories $\Comp$ and $\T$);
\item is {\em continuous} if $F(\Comp)\subset\Comp$ and $F$ preserves the limits of inverse spectra in the category $\Comp$;
\item {\em preserves intersections} if for any compact Hausdorff space $X$ and closed subsets $X_\alpha\subset X$, $\alpha\in A$, with intersection $Z=\bigcap_{\alpha\in A}X_\alpha$, we get $Fi^Z_X(Z)=\bigcap_{\alpha\in A}Fi^{X_\alpha}_X(FX_\alpha)$.
\end{itemize}
Here for two compact Hausdorff spaces $A\subset B$ by $i^A_B:A\to B$ we denote the identity embedding.
\smallskip

Theorem~\ref{support} is a key ingredient in the proof of the following:

\begin{theorem}\label{main} Each monomorphic functor $F:\Comp\to\T$ with finite supports is epimorphic and its maximal $\emptyset$-modification $F^\circ:\Comp\to\T$ preserves intersections.
\end{theorem}

For endofunctors $F:\Comp\to\Comp$ in the category of compacta we can prove a bit more:

\begin{theorem}\label{main2} For each monomorphic functor $F:\Comp\to\Comp$ with finite supports its maximal $\emptyset$-modification $F^\circ:\Comp\to\Comp$ is a monomorphic, epimorphic, continuous,  intersection preserving functor with finite supports. Moreover, the functors $F$ and $F^\circ$ preserve the weight of infinite compacta if and only if for every $n\in\w$ the space $Fn$ is metrizable.
\end{theorem}

In \cite{Bas} V.Basmanov proved that each monomorphic continuous functor $F:\Comp\to\Comp$ of finite degree $\deg F\le n$ preserves (finite-dimensional) compact ANRs provided $F$ preserves intersections and the spaces $F\emptyset$ and $Fn$ are finite-dimensional ANRs. Theorem~\ref{main2} allows us to improve this Basmanov's result:

\begin{theorem} A monomorphic functor $F:\Comp\to \Comp$ of finite degree $\deg F\le n$ preserves (finite-dimensional) compact ANRs provided $F\emptyset$, $ F^\circ\emptyset$, and $Fn$ are finite-dimensional ANRs.
\end{theorem}

This theorem implies the following corollary that will be applied in \cite{BH} for studying the functors of  free topological universal algebras.

\begin{corollary} A monomorphic functor $F:\Comp\to \Comp$ of finite degree $\deg F\le n$ preserves (finite-dimensional) compact ANRs provided the space $F1$ is finite and $Fn$ is a finite-dimensional ANR.
\end{corollary}

\section{Proof of Theorem~\ref{support}}\label{s2}

We assume that $F:\C\to\T$ is a monomorphic functor with finite supports defined on a topological category $\C$ and $F^\circ:\C\to\T$ is its maximal $\emptyset$-modification. We recall that for a finite subset $A$ of a topological space $X$ by $i^A_X:A\to X$ we denote the identity map from $A$ endowed with the discrete topology to $X$.

Theorem~\ref{support} will be derived from the following lemma.

\begin{lemma}\label{l1} For any subsets  $A,B$ of a  finite discrete space $X$ we get $$F^\circ i^{A\cap B}_X(F^\circ(A\cap B))=F^\circ i^A_X(FA)\cap F^\circ i^B_X(FB).$$
\end{lemma}

\begin{proof} The inclusion $F^\circ i^{A\cap B}_X(F^\circ(A\cap B))\subset F^\circ i^A_X(F^\circ A)\cap F^\circ i^B_X(F^\circ B)$ follows from the functoriality of $F^\circ$. To prove the reverse inclusion, we consider 4 cases.
\smallskip

1. If $A\subset B$, then $i^A_X=i^B_X\circ i^A_B$ and then $F^\circ i^A_X(F^\circ A)=F^\circ i^B_X\circ F^\circ i^A_B(F^\circ A)\subset F^\circ i^B_X(F^\circ B)$ and
$$F^\circ i^A_X(F^\circ A)\cap F^\circ i^B_X(F^\circ B)=F^\circ i^A_X(F^\circ A)=F^\circ i^{A\cap B}_X(F^\circ(A\cap B)).$$

2. By analogy we can consider the case $B\subset A$. These two cases happen if one of the sets $A$ or $B$ is empty.
\smallskip

3. The sets $A,B\subset X$ are non-empty but have empty intersection $A\cap B=\emptyset$. In this case $F^\circ A=FA$ and $F^\circ B=FB$. To prove that $Fi^A_X(FA)\cap Fi^B_X(FB)\subset F^\circ i^\emptyset_X(F^\circ\emptyset)$, fix any element $c\in Fi^A_X(FA)\cap Fi^B_X(FB)$.
We need to prove that $c\in F^\circ i^\emptyset_X(F^\circ\emptyset)$.
Find elements $c_A\in FA$ and $c_B\in FB$ such that $Fi^A_X(c_A)=c=Fi^B_X(c_B)$.

First we prove that for any point $a\in A$ we get $c\in Fi^{\{a\}}_X(F\{a\})\subset FX$. Indeed, consider the map $r:X\to A$ such that $r(x)=x$ if $x\in A$ and $r(x)=a$ if $x\in X\setminus A$. Let $r^B_{\{a\}}:B\to\{a\}$ denote the constant map and observe that $i^A_X\circ r\circ i^B_X=i^{\{a\}}_X\circ r^B_{\{a\}}$.

Applying the functor $F$ to the equality $i^A_X=i^A_X\circ r\circ i^A_X$, we get
$c=Fi^A_X(c_A)=Fi^A_X\circ Fr\circ Fi^A_X(c_A)=Fi^A_X\circ Fr(c)=Fi^A_X\circ Fr\circ Fi^B_X(c_B)\in F(i^A_X\circ r\circ i^B_X)(c_B)=
F(i^{\{a\}}_X\circ r^B_{\{a\}})(c_B)=Fi^{\{a\}}_X(Fr^B_{\{a\}}(c_B))\in Fi^{\{a\}}_X(F\{a\})\subset FX$.

By the same argument, we can prove that $c\in Fi^{\{b\}}_X(F\{b\})\subset FX$ for any $b\in B$.

Let $r^X_1:X\to 1$ be the unique map and $f_a,f_b:1\to X$ be two maps such that $f_a(0)=a\in A$ and $f_b(0)=b\in B$. Since $c\in Fi^{\{a\}}_X(F\{a\})=Ff_a(F1)$ and $c\in Fi^{\{b\}}_XF(\{b\})=Ff_b(F1)$ there are two elements $c_a,c_b\in F1$ such that $Ff_a(c_a)=c=Ff_b(c_b)$. Since $r^X_1\circ f_a=\id=r^X_1\circ f_b$, we conclude that
$$c_a=Fr^X_1\circ Ff_a(c_a)=Fr^X_1(c)=Fr^X_1\circ Ff_b(c_b)=c_b.$$ Now we see that the element $c_1=c_a=c_b$ belongs to $F^\circ\emptyset$ and $c=Ff_a(c_1)=Ff_b(c_1)$, which means that $c=F^\circ i^1_X\,(c_1)\in F^\circ i^\emptyset_X(F^\circ\emptyset)$ according to the definition of the morphism $F^\circ i^\emptyset_X:F^\circ\emptyset\to F^\circ\! X=FX$.
\smallskip

4. The intersection $A\cap B$ is not empty. In this case $F^\circ\! A=FA$, $F^\circ\! B=FB$ and $F^\circ (A\cap B)=F(A\cap B)$.

To prove that $Fi^A_X(FA)\cap Fi^B_X(FB)\subset Fi^{A\cap B}_X(F(A\cap B))$, fix any element $c\in Fi^A_X(FA)\cap Fi^A_X(FB)$ and find  elements $c_A\in FA$ and $c_B\in FB$ such that $Fi^A_X(c_A)=c=Fi^B_X(c_B)$.

Choose any map $r^X_{A\cap B}:X\to A\cap B$ such that $r(x)=x$ for all $x\in A\cap B$ and define retractions $r^X_A:X\to A$ and $r^X_B:X\to B$ by
$$
r^X_A(x)=
\begin{cases}
x&\mbox{if $x\in A$}\\
r^X_{A\cap B}(x)&\mbox{otherwise}
\end{cases}
\quad\mbox{ and }\quad
r^X_B(x)=
\begin{cases}
x&\mbox{if $x\in B$}\\
r^X_{A\cap B}(x)&\mbox{otherwise}.
\end{cases}
$$
Observe that $r^X_{A\cap B}=r^X_B\circ r^X_A=r^X_A\circ r^X_B$.

We claim that $c_A=Fr^X_A(c)$. Since $i^A_X=i^A_X\circ r^X_A\circ i^A_X$, we get
$$Fi^A_X(c_A)=Fi^A_X\circ Fr^X_A\circ Fi^A_X(c_A)=Fi^A_X\circ Fr^X_A(c)=Fi^A_X(Fr^X_A(c))$$and hence $c_A=Fr^A_X(c)$ by the injectivity of the map $Fi^A_X:FA\to FX$.

The same argument yields that $c_B=Fr^X_B(c)$.
Now consider the element $c_{AB}=Fr^X_{A\cap B}(c)\in F(A\cap B)$. Since $r^X_{A\cap B}=r^X_{A\cap B}\circ i^A_X\circ r^X_A$, we get $$c_{AB}=Fr^X_{A\cap B}(c)=Fr^X_{A\cap B}\circ Fi^A_X\circ Fr^X_A(c)=Fr^X_{A\cap B}\circ Fi^A_X(c_A).$$

Applying the functor $F$ to the equality $i^{A\cap B}_B\circ r^X_{A\cap B}\circ i^A_X=r^X_B\circ i^A_X$, we get
$$Fi^{A\cap B}_B(c_{AB})=Fi^{A\cap B}_B\circ Fr^X_{A\cap B}\circ Fi^A_X(c_A)=Fr^X_B\circ Fi^A_X(c_A)=Fr^X_B(c)=c_B$$and then
$$Fi^{A\cap B}_X(c_{AB})=F(i^B_X\circ i^{A\cap B}_B)(c_{AB})=Fi^B_X\circ Fi^{A\cap B}_B(c_{AB})=Fi^B_X(c_B)=c, $$ which means that $c=Fi^{A\cap B}_X(c_{AB})\in Fi^{A\cap B}_X(F(A\cap B))$.
\end{proof}

The following lemma implies Theorem~\ref{support}.

\begin{lemma}\label{l2} For any object $X$ of the category $\C$ and an element $a\in F^\circ X$ the support $A=\supp_X(a)$ is a well-defined finite subset of $X$ such that $a\in F^\circ i^A_X(F^\circ A)$.
\end{lemma}

\begin{proof} We recall that $\supp_X(a)=\cap\mathcal B$ where $\mathcal B=\{B\in[X]^{<\w}:a\in F^\circ i^B_X(F^\circ B)\}$. First we show that the family $\mathcal B$ is not empty.
Since the functor $F^\circ$ has finite supports, there is a map $f:C\to X$ from a finite discrete space $C$ such that $a\in F^\circ f(F^\circ C)$. Let $B=f(C)$ and $f^C_B:C\to B$ be the map such that $f^C_B(c)=f(c)$ for all $c\in C$. Since $f=i^B_X\circ f^C_B$, we get $F^\circ f=F^\circ i^B_X\circ F^\circ\! f^C_B$ and
$$a\in F^\circ f(F^\circ C)=F^\circ(i^B_X\circ f^C_B)(F^\circ C)=F^\circ i^B_X(F^\circ\! f^C_B(F^\circ C))\subset F^\circ i^B_X(F^\circ B).$$

Now we see that $B\in\mathcal B$ and the family $\mathcal B$ is not empty. So, the intersection $\supp(a)=\cap\mathcal B$ is a well-defined finite subset of $X$. Since $\supp(a)=\cap\mathcal B$ is finite, there exist subsets $B_1,B_2,\dots, B_n\in\mathcal B$ of $X$ such that $\supp(a)=\bigcap_{i=1}^n B_i$. For every $k\le n$ let $A_k=\bigcap_{i=1}^kB_i$.
Thus $A_1=B_1$ and $A_n=\supp(a)$.

We claim that $a\in F^\circ i^{A_k}_X(F^\circ A_k)$ for every $1\le  k\le n$. This will be done by induction on $k$. For $k=1$ this inclusion follows from $A_1=B_1$ and the choice of $B_1$. Assume that $a\in F^\circ i^{A_{k-1}}_X(F^\circ A_{k-1})$
for some $k\le n$. Taking into account that $A_k=A_{k-1}\cap B_k$ and $a\in F^\circ i^{B_k}_X(F^\circ B_k)$ and applying Lemma~\ref{l1}, we conclude that
$a\in F^\circ i^{A_{k-1}}_X(F^\circ A_{k-1})\cap
F^\circ i^{B_{k}}_X(F^\circ B_{k})=
F^\circ i^{A_{k}}_X(F^\circ A_{k})$.

For $k=n$ we get $A_n=\supp(a)$ and hence $a\in F^\circ i^{A_{n}}_X(F^\circ A_{n})$.
\end{proof}

\section{Proof of Theorem~\ref{main}}

Let $F:\Comp\to\T$ be a monomorphic functor with finite supports in the category of compacta and $F^\circ:\Comp\to\T$ be its maximal $\emptyset$-modification. By Theorem~\ref{t1}, the functor $F^\circ$ is monomorphic. Also it is clear that $F^\circ$ has finite supports. The remaining two properties of $F$ and $F^\circ$ stated in Theorem~\ref{main} are proved in the following two  lemmas.

\begin{lemma}\label{l3} Each monomorphic functor $F:\Comp\to\T$ with finite supports preserves surjective maps and hence is epimorphic.
\end{lemma}

\begin{proof}  Let $f:X\to Y$ be a surjective map between compact spaces and $b\in FY$ be any element. Since $F$ has finite supports, there is a finite subset $B\subset Y$ such that $b\in Fi^B_Y(FB)$ where $i^B_Y:B\to Y$ is the identity map from $B$ to $Y$.  Let $s:B\to X$ be any map such that $f\circ s=i^B_Y$. Such a map $s$ exists because the map $f$ is surjective. Fix an element  $b_B\in FB$ such that $b=Fi^B_X(b_B)$ and let $a=Fs(b_B)$. Applying the functor $F$ to the equality $f\circ s=i^B_X$, we get $b=Fi^B_X(b_B)=Ff\circ Fs(b_B)=Ff(a)$, witnessing that the map $Ff:FX\to FY$ is surjective. Therefore $F$ is an epimorphic functor.
\end{proof}

\begin{lemma}\label{l4} The functor $F^\circ:\Comp\to\T$ preserves intersections.
\end{lemma}

\begin{proof} Let $X$ be a compact Hausdorff space and $X_\alpha$, $\alpha\in A$, be closed subspaces of $X$ with intersection $Z=\bigcap_{\alpha\in A}X_\alpha$.
For two compact Hausdorff spaces $A\subset B$ by  $i^A_B:A\to B$ we denote the identity embedding.

We need to prove that $F^\circ i^Z_X(F^\circ Z)=\bigcap_{\alpha\in A}F^\circ i^{X_\alpha}_X(F^\circ X_\alpha)$. The inclusion $$F^\circ i^Z_X(F^\circ Z)\subset\bigcap_{\alpha\in A}F^\circ i^{X_\alpha}_X(F^\circ X_\alpha)$$ trivially follows from the functoriality of $F^\circ$.

In order to prove the reverse inclusion, fix any element $b\in\bigcap_{\alpha\in A}F^\circ i^{X_\alpha}_X(F^\circ \! X_\alpha)$. For every $\alpha\in A$ find an element $b_\alpha\in F^\circ X_\alpha$ such that $b=F^\circ i^{X_\alpha}_X(b_\alpha)$. Since the functor $F^\circ$ has finite supports, there is a finite set $Y_\alpha\subset X_\alpha$ such that $b_\alpha\in F^\circ i^{Y_\alpha}_{X_\alpha}(F^\circ Y_\alpha)$. Since $i^{Y_\alpha}_X=i^{X_\alpha}_X\circ i^{Y_\alpha}_{X_\alpha}$, we get
$$b=F^\circ i^{X_\alpha}_X(b_\alpha)\in F^\circ i^{X_\alpha}_X(F^\circ i^{Y_\alpha}_{X_\alpha}(F^\circ Y_\alpha))=F^\circ i^{Y_\alpha}_X(F^\circ Y_\alpha).$$

The definition of the set $A=\supp(b)$ guarantees that $A=\supp(b)\subset Y_\alpha\subset X_\alpha\subset X$. Then $A\subset\bigcap_{\alpha\in A}X_\alpha=Z$ and $i^A_X=i^Z_X\circ i^A_Z$. By Theorem~\ref{support}, $b\in F^\circ i^A_X(F^\circ A)$ and consequently, there is an element $a\in F^\circ A$ such that $b=F^\circ i^A_X(a)$. Let $c=F^\circ i^A_Z(a)\in F^\circ Z$. Then
$$b=F^\circ i^A_X(a)=F^\circ(i^Z_X\circ i^A_Z)(a)=F^\circ i^Z_X(F^\circ i^A_Z(a))=F^\circ i^Z_X(c)\in F^\circ i^Z_X(F^\circ Z),$$which completes the proof.
\end{proof}

\section{Proof of Theorem~\ref{main2}}

Let $F:\Comp\to\Comp$ be a monomorphic functor with finite supports in the category of compacta. By Theorem~\ref{main}, its maximal $\emptyset$-modification $F^\circ:\Comp\to\Comp$ is a monomorphic, epimorphic functor with finite supports, which preserves intersections.  The remaining two properties of $F^\circ$ stated in Theorem~\ref{main2} are proved in the following two lemmas.

\begin{lemma}\label{l5} Each monomorphic functor $F:\Comp\to\Comp$ with finite supports is continuous.
\end{lemma}

\begin{proof} By Lemma~\ref{l3}, $F$ is epimorphic. By Theorem~2.2.2 of \cite{TZ} the continuity of the functor $F$ will follow as soon as we check that for each cardinal $\kappa$ and any two distinct elements $a,b\in F(\II^\kappa)$ there is a finite subset $D\subset\kappa$ such that $F p_D(a)\ne F p_D(b)$ where $p_D:\II^\kappa\to \II^D$ is the projection of the Tychonov cube $\II^\kappa$ onto its face $\II^D$.

Since $F$ has finite supports, there is a finite subset $C\subset \II^\kappa$ such that $a,b\in Fi^C(FC)$ where $i^C:C\to \II^\kappa$ denotes the identity embedding. Find elements $a_C,b_C\in F C$ such that $a=F i^C(a_C)$ and $b=F i^C(b_C)$. Since $C$ is finite, we can find a finite subset $D\subset\kappa$ such that the composition $p_D\circ i^C:C\to I^D$ is injective. Since $F$ is monomorphic, the map $F p_D\circ F i^C:FC\to F \II^D$ is injective and hence
$$F p_D(a)=F p_D\circ F i^C(a_C)\ne F p_D\circ F i^C(b_C)=F p_D(b).$$
\end{proof}

For a topological space $X$ by $w(X)$ we denote its weight (equal to the smallest cardinality of the base of the topology of $X$). For two compact Hausdorff spaces $X,Y$ by $C(X,Y)$ we denote the space of continuous functions from $X$ to $Y$, endowed with the compact-open topology.

\begin{lemma}\label{l2} If $F:\Comp\to\Comp$ is a monomorphic functor with finite supports, then $w(FX)\le\sup\{w(X),w(Fn):n\in\w\}$ for each infinite compact space $X$.
\end{lemma}

\begin{proof} By Lemmas~\ref{l3} and \ref{l5}, the functor $F$ is epimorphic and continuous. Then by Theorem 2.2.3 of \cite{TZ}, for every $n\in\w$ the map
$$F:C(n,X)\to C(Fn,FX),\;\;F:f\mapsto Ff,$$
is continuous and so is the map $$\xi_n:C(n,X)\times Fn\to FX,\;\;\xi_n:(f,a)\mapsto Ff(a),$$
according to the exponential law for the compact-open topology \cite[3.4.8]{En}. Then the image $F_nX=\xi_n(C(n,X)\times Fn)\subset FX$ is a compact space of weight
$$w(F_nX)\le w(C(n,X)\times Fn)\le \max\{w(X^n),w(Fn)\}=\max\{w(X),w(Fn)\},$$
see \cite[3.1.22]{En}.

Since $F$ has finite supports, the compact space $FX$ is equal to the countable union $FX=\bigcup_{n\in\w}F_nX$ and hence has weight
$$w(FX)\le\sup_{n\in\w}w(F_nX)\le\sup\{w(X),w(Fn):n\in\w\}$$
according to \cite[3.1.20]{En}.
\end{proof}

\end{document}